\documentclass{article}
\usepackage[utf8]{inputenc}
\usepackage{babel}
\usepackage{amsmath,amsthm,amssymb}
\usepackage[shortlabels]{enumitem}
\usepackage{hyperref,cleveref}
\usepackage{accents}
\usepackage{xcolor}
\usepackage{tikz}

\usepackage{fullpage}

\newcommand{\bivec}[1]{\accentset{\leftrightarrow}{#1}}
\newcommand{\bigslant}[2]{{\raisebox{.2em}{$#1$}\left/\raisebox{-.2em}{$#2$}\right.}}

\title{Zero sum cycles in 
complete digraphs}
\author
{
Tam\'{a}s M\'{e}sz\'{a}ros\thanks{Berlin, Email: \texttt{tmeszaros87@gmail.com}.}
\and
Raphael Steiner\thanks{Institut für Mathematik, Technische Universit\"at Berlin, Germany. Funded by DFG-GRK 2434 Facets of Complexity. Email: \texttt{steiner@math.tu-berlin.de}.}
}
\date{\today}

\newtheorem{thm}{Theorem}
\newtheorem{corollary}{Corollary}

\newtheorem{lemma}{Lemma}

\newtheorem{obs}{Observation}
\newtheorem{conj}{Conjecture}
\newtheorem{ques}{Question}

\begin{document}

\maketitle

\begin{abstract}
Given a non-trivial finite Abelian group $(A,+)$, let $n(A) \ge 2$ be the smallest integer such that for every labelling of the arcs of the bidirected complete graph $\bivec{K}_{n(A)}$ with elements from $A$ there exists a directed cycle 
for which the sum of the arc-labels is zero. The problem of determining $n(\mathbb{Z}_q)$ for integers $q \ge 2$ was recently considered by Alon and Krivelevich~\cite{AK20}, who proved that $n(\mathbb{Z}_q)=O(q \log q)$. Here we improve their result and show that $n(\mathbb{Z}_q)$ grows linearly. More generally we prove that for every finite Abelian group $A$ we have $n(A) \le 8|A|$, while if $|A|$ is prime then $n(A) \le \frac{3}{2}|A|$. 

As a corollary we also obtain that every $K_{16q}$-minor contains a cycle of length divisible by $q$ for every integer $q \ge 2$, which improves a result from~\cite{AK20}. 
\end{abstract}

\section{Introduction}

Zero-sum problems are a branch of Ramsey theory with an algebraic flavour. One of the earliest results of this form is the Erd\H{o}s-Ginzburg-Ziv Theorem~\cite{EGZ61} which is considered to be the common ancestor of many zero-sum problems. It states that if $k\mid m$ then among $m+k-1$ integers one can always find $m$ whose sum is divisible by $k$. Over time many problems of this and similar type have been considered, and zero-sum problems have established themselves as a well studied and substantial branch of contemporary combinatorics.

A much studied example of a zero-sum problem for graphs goes as follows. What is the smallest number $n$ so that any complete graph with edges labelled by the elements of a finite group $G$ contains a subgraph of a prescribed
type in which the total weight of the edges is 0 in $G$? For examples and an overview of such results, the interested reader may consult e.g.~\cite{AL89,AC93} and the survey article~\cite{C96}. Here we consider this problem for complete directed graphs, where the desired subgraph is a directed cycle. This scenario was recently also considered by Alon and Krivelevich~\cite{AK20}.

Given an integer $n \ge 1$, we denote by $\bivec{K}_n$ the complete digraph consisting of $n$ vertices and whose arc-set consists of all ordered pairs of vertices. Given a set $A$, for us an $A$-arc-labeling of $\bivec{K}_n$ is simply a function $w:A(\bivec{K}_n) \rightarrow A$. 
If $(A,+)$ is an Abelian group, we say that $w$ is \emph{zero-sum-free} if there is no directed cycle for which the sum of arc-labels is zero. In this paper, for a non-trivial finite Abelian group $(A,+)$, we are interested in determining the smallest integer $n(A)\geq 2$ such that $\bivec{K}_{n(A)}$ has no zero-sum-free $A$-arc-labeling, i.e. for every $A$-arc-labeling there is a directed cycle for which the sum of the arc-labels is zero.

In~\cite{AK20} Alon and Krivelevich proved that $n(\mathbb{Z}_q)\leq \lceil2q\ln q\rceil$ for every integer $q\ge 2$ and $n(\mathbb{Z}_p)\leq 2p-1$ for every prime number $p$. Here we improve their results and show that $n(A)$ grows at most linearly for every non-trivial finite Abelian group $A$. 
\begin{thm}\label{thm:main1}
For every non-trivial finite Abelian group $(A,+)$ we have $n(A) \le 8|A|$.
\end{thm}
Our second main result improves the upper bound of~\cite{AK20} in the case of primes.
\begin{thm}\label{thm:main2}
For every prime $p\geq 3$ we have $n(\mathbb{Z}_p) \le \frac{3p-1}{2}$.
\end{thm}
%

The motivation of Alon and Krivelevich for studying the function $n(\mathbb{Z}_q)$ came from a related problem about the containment of cycles with 
particular lengths in minors of complete graphs. Our improvement on the upper bound on $n(\mathbb{Z}_q)$ also directly improves the corresponding result from \cite{AK20}. 
\begin{corollary}\label{cor:easy}
Let $q \ge 2$ be an integer. Then every $K_{2n(\mathbb{Z}_q)}$-minor contains a cycle of length divisible by $q$. In particular, every $K_{16q}$-minor contains a cycle of length divisible by $q$, and if $p$ is a prime, then every $K_{3p-1}$-minor contains a cycle of length divisible by $p$. 
\end{corollary}

\section{Proofs}

Let $n \ge 2$ be an integer and $(A,+)$ a non-trivial finite Abelian group. We say that an $A$-arc labeling $w$ of $\bivec{K}_n$ is \emph{$A$-complete at the vertices $u,v\in V(\bivec{K}_n)$} if for every $a\in A$ there is a directed path $P_a$ from $u$ to $v$ such that $\sum_{(x,y)\in P_a}w(x,y)=a$. Given a vertex $v$ of $\bivec{K}_n$ and an element $c \in A$ 
let $w'$ 
be the $A$-arc-labelling defined by 
\begin{equation*}
w'(x,y)=\begin{cases}
w(x,y) & x, y \neq v \cr
w(x,y)+c & x=v \cr
w(x,y)-c & y=v
\end{cases}.   
\end{equation*}
This operation on arc labellings is called a 
\emph{switching} \emph{by} $c$ \emph{at} $v$, and is denoted by $S_{c,v}$. Two $A$-arc-labellings of $\bivec{K}_n$ are \emph{switching-equivalent} if we can obtain one from the other by a sequence of switchings. The following properties of switchings follow directly from the definitions, and will be very important for our investigations.
\begin{obs}\label{obs:switching}
Let 
$w$ and $w'$ be switching-equivalent $A$-arc-labellings of $\bivec{K}_n$, then the following hold. 
\begin{itemize}
    \item $w$ is zero-sum-free if and only if $w'$ is so.
    \item $w$ is $A$-complete at $u,v$ if and only if $w'$ is so.
\end{itemize}
\end{obs}
%
%
The following lemma is our main technical result on the way to the proof of Theorem~\ref{thm:main1}.
\begin{lemma}\label{lemma:main}
Let $n \ge 2$, $(A,+)$ a non-trivial finite Abelian group and $w$ 
an $A$-arc-labelling of $\bivec{K}_n$. Then there exists another $A$-arc-labelling $w'$ which is switching-equivalent to $w$ and at least one of the following holds.
\begin{itemize}
    \item There exists a vertex set $V \subseteq V(\bivec{K}_n)$ of size $|V| \ge n-4|A|$ and a proper subgroup $B < A$ such that 
    $w'(x,y)\in B$ for every $(x,y) \in A(\bivec{K}_n[V])$.
    \item There exist vertices $u,v \in V(\bivec{K}_n)$ such that $w'$ is $A$-complete at $u,v$. 
\end{itemize}
\end{lemma}
\begin{proof}
Suppose towards a contradiction that the claim was false, and let 
$A$ be a group of smallest size for which it fails. In particular, there exists some $n \in \mathbb{N}$ and an $A$-arc-labelling $w$ 
such that the following hold for every $A$-arc-labelling $w'$ that is switching-equivalent to $w$.
\begin{enumerate}[label=(\roman*)]
    \item For every $V \subseteq V(\bivec{K}_n)$ with $|V|\ge n-4|A|$ 
    and for every proper subgroup $B<A$ 
    there exists $(x,y)\in A(\bivec{K}_n[V])$ such that $w'(x,y)\notin B$. 
    \item For every $u,v \in V(\bivec{K}_n)$ there exists some $a \in A$ such that there is no directed path 
    $P_a$ in $\bivec{K}_n$ from $u$ to $v$
    such that $\sum_{(x,y) \in A(P_a)}{w'(x,y)}=a$. 
\end{enumerate}
In the following we will lead these assumptions towards a contradiction. To do so, we first prove the following auxiliary claim by induction on $i$.
\paragraph{Claim.} For every $0 \le i \le |A|-1$ 
there exist pairwise distinct vertices $x_0, y_0,x_1,y_1\ldots,x_{i-1}, y_{i-1},x_{i}\in V(\bivec{K}_n)$ and an $A$-arc-labelling $w_i$ of $\bivec{K}_n$ that is switching-equivalent to $w$ such that the following hold.
\begin{itemize}
    \item $w_i(x_{j-1},x_j)=0$ for every $1 \le j \le i$.
    \item Let $a_{i,j}:=w_i(x_{j-1},y_{j-1})+w_i(y_{j-1},x_j)$ for $1 \le j \le i$, and 
    \begin{equation*}
         A_i:=\left\{\sum_{j \in J}{a_{i,j}}
         \ |\ J \subseteq \{1,\ldots,i\}\right\} \subseteq A.
    \end{equation*}
Then $|A_i|\geq i+1$. 
\end{itemize}

Before moving on to the proof of the claim, note that the sequence $x_0, y_0,x_1,y_1\ldots,x_{i-1}, y_{i-1},x_{i}$ represents a chain of triangles connecting $x_0$ and $x_i$, see Figure~\ref{fig:chain}. When going from $x_0$ to $x_i$, at every intermediate vertex $x_{j-1}$ we have the choice to either go directly to $x_{j}$ and pick up the label $0$, or to take the detour $x_{j-1}\rightarrow y_{j-1}\rightarrow x_{j}$ and pick up labels summing up to $a_{i,j}$. Therefore, the set $A_i$ represents all possible arc-label sums along possible paths from $x_0$ to $x_i$ in this configuration. 

\begin{figure}
    \centering
    \begin{tikzpicture}[scale=.73,main node/.style={circle,draw,color=black,fill=black,inner sep=0pt,minimum width=4pt}]
        \node[main node] (x0) at (0,0) [label=above:$x_0$]{};
	    \node[main node] (y0) at (2,-2) [label=below:$y_0$]{};
	    \node[main node] (x1) at (4,0) [label=above:$x_1$]{};
	    \node[main node] (y1) at (6,-2) [label=below:$y_1$]{};
	    \node[main node] (x2) at (8,0) [label=above:$x_2$]{};
	    \node[main node] (xi-1) at (16,0) [label=above:$x_{i-1}$]{};
	    \node[main node] (yi-1) at (18,-2) [label=below:$y_{i-1}$]{};
	    \node[main node] (xi) at (20,0) [label=above:$x_i$]{};
	    
	    \draw[->] (x0)--(x1) node[pos=.5,above] {$0$};
	    \draw[->] (x0)--(y0);
	    \draw[->] (y0)--(x1);
	    
	    \draw[->] (x1)--(x2) node[pos=.5,above] {$0$};
	    \draw[->] (x1)--(y1);
	    \draw[->] (y1)--(x2);
	    
	    \draw[->] (xi-1)--(xi) node[pos=.5,above] {$0$};
	    \draw[->] (xi-1)--(yi-1);
	    \draw[->] (yi-1)--(xi);
	    
	    \draw[loosely dotted,thick] (8.5,-1)--(15.5,-1);
	    
	    \def \radius {1.5cm}
	    
	    \draw[->] (0.92,-0.8)
         arc ({-135}:{-45}:\radius);
        \node[] at (2,-0.4) [label=below:$a_{i,1}$]{};
        
        \draw[->] (x1)+(0.92,-0.8)
         arc ({-135}:{-45}:\radius);
        \node[] at (6,-0.4) [label=below:$a_{i,2}$]{};
        
        \draw[->] (xi-1)+(0.92,-0.8)
         arc ({-135}:{-45}:\radius);
        \node[] at (18,-0.4) [label=below:$a_{i,i}$]{};
    \end{tikzpicture}
    \caption{A chain of triangles connecting the vertices $x_0$ and $x_i$.}
    \label{fig:chain}
\end{figure}
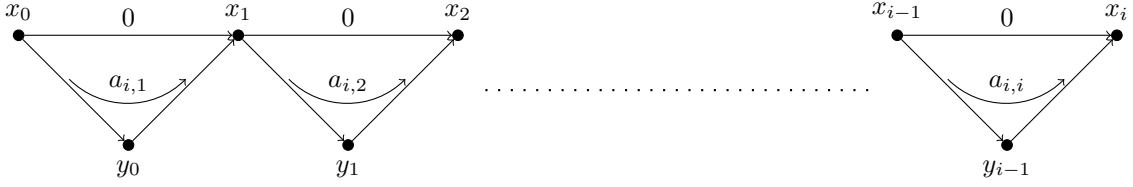

\begin{proof}[Proof of the claim.]
For $i=0$ note that, by definition, $A_i$ always includes $0$ (the empty sum), so the claim holds trivially. Moving on to the inductive step, suppose $1 \le i \le |A|-1$ and the claim holds for $i-1$. Put $V:=V(\bivec{K}_n)\setminus \{x_0, y_0,x_1,y_1\ldots,x_{i-2}, y_{i-2},x_{i-1}\}$. Then $|V|= n-(2i-1) \ge n-2|A|+3$. Let $w_i$ be the $A$-arc-labelling of $\bivec{K}_n$ obtained from $w_{i-1}$ by switching by $w_{i-1}(x_{i-1},y)$ at every vertex $y \in V$. Then $w_{i}$ is clearly switching-equivalent to $w_{i-1}$ and hence to $w$, and by definition satisfies $w_i(x_{j-1},x_j)=w_{i-1}(x_{j-1},x_j)=0$ for $1 \le j \le i-1$ as well as $w_i(x_{i-1},y)=0$ for every $y \in V$. 

\smallskip

We distinguish two cases.

\paragraph{Case 1.} There exists a non-trivial subgroup $\{0\}\neq B\le A$ 
such that $B\subseteq A_{i-1}$. 

\smallskip

Let $H$ be the quotient group $\bigslant{A}{B}$ and let 
us define an $H$-arc-labelling $\tilde{w}$ of the complete subdigraph $\bivec{K}_n[V]$ by putting $\tilde{w}_i(x,y)=w_i(x,y)+B$ for all $x,y \in V$. We call $\tilde{w}_i$ the \emph{$B$-factor of $w_i$ restricted to $\bivec{K}_n[V]$}. 
By our initial minimality assumption on $|A|$, the claim of the lemma has to hold for $H$, the complete digraph $\bivec{K}_n[V]$ and the $H$-arc-labelling $\tilde{w}_i$. It follows that there exists an $H$-arc-labelling $\tilde{w}_i'$ of $\bivec{K}_n[V]$ switching-equivalent to $\tilde{w}_i$ such that at least one of the following holds.
\begin{itemize}
    \item[(i')] There exists a vertex set $W \subseteq V$ of size $|W| \ge |V|-4|H|$ and a proper subgroup $\tilde{G}<H$ such that 
    $\tilde{w}_i'(x,y)\in \tilde{G}$ 
    for every $(x,y) \in A(\bivec{K}_n[W])$. 
    \item[(ii')] There exist vertices $u',v' \in V(\bivec{K}_n)$ such that $\tilde{w}_i'$ is $H$-complete at $u',v'$. 
\end{itemize}

Since $\tilde{w}_i$ and $\tilde{w}_i'$ are switching-equivalent $H$-arc-labellings on $\bivec{K}_n[V]$, there exists a sequence $S_{c_i+B,v_i}$, $i\in I$ of switchings that transform $\tilde{w}_i$ to $\tilde{w}_i'$. Applying the corresponding sequence $S_{c_i,v_i}$, $i\in I$ of swichings to $w_i$ one obtains an $A$-arc-labelling $w_i'$ of $\bivec{K}_n$ which is switching-equivalent to $w_i$, and hence to $w$, and whose $B$-factor restricted to $\bivec{K}_n[V]$ is $\tilde{w}_i'$. We next consider two cases depending on whether (i') or (ii') occurs.

\paragraph{Case 1 + (i') holds.} As $B$ is a non-trivial subgroup of $A$ 
, we have $|H|\leq |A|/2$ 
and hence $|W| \ge |V|-4|H| \ge |V|-2|A|\ge n-2|A|+3-2|A| \ge n-4|A|$. Let $G$ be the pre-image of $\tilde{G}$ under the canonical group homomorpism $A\rightarrow H=\bigslant{A}{B}$. Then $G$ is a proper subgroup of $A$, and by the assumption on $\tilde{w}_i'$ and the fact that $\tilde{w}_i'$ is the $B$-factor of $w_i'$ restricted to $\bivec{K}_n[V]$, we find that $w_i'(x,y)\in G$ 
for every $(x,y) \in A(\bivec{K}_n[W])$. 
The existence of $W$ and $w_i'$ yields a contradiction to property (i).

\paragraph{Case 1 + (ii') holds.} As $\tilde{w}_i$ and $\tilde{w}_i'$ are switching equivalent, by Observation~\ref{obs:switching} (ii') has to hold for $\tilde{w}_i$ as well. 
Then, on the one hand, by property (ii), when applied to the vertices $x_0,v' \in V(\bivec{K}_n)$ and the $A$-arc-labelling $w_i$
, we find that there exists $a \in A$ such that there is no directed $x_0,v'$-dipath $P_a$ in $\bivec{K}_n$ such that $\sum_{(x,y) \in A(P_a)}{w_i(x,y)}=a$. 
On the other hand, 
by the $H$-completeness of $\tilde{w}_i$, there is a $u',v'$-dipath $P_{a+B}$ in $\bivec{K}_n[V]$ for which we have $\sum_{(x,y) \in A(P_{a+B})}{\tilde{w}_i(x,y)}=a+B$. 
Since 
$\tilde{w}_i$ is the $B$-factor of $w_i$ restricted to $\bivec{K}_n[V]$, it follows that there exists an element $b\in B$ such that 
\begin{equation*}
  \sum_{(x,y) \in A(P_{a+B})}{w_i(x,y)}=a+b.  
\end{equation*}
By our assumption $B\subseteq A_{i-1}$, so we can find $J \subseteq \{1,...,i-1\}$ such that
\begin{equation*}
    \sum_{j \in J}{a_{i-1,j}}= -b. 
\end{equation*}
Let now $P$ be the dipath in $\bivec{K}_n$ obtained by concatenating the arcs $(x_{j-1},x_j)$ for $j \in \{1,\ldots,i-1\}\setminus J$, the three-vertex-dipaths $(x_{j-1},y_{j-1}), (y_{j-1},x_j)$ for $j \in J$, the arc $(x_{i-1},u')$ and the dipath $P_{a+B}$. Then $P$ starts at $x_0$, ends at $v'$, and satisfies
\begin{equation*}
    \sum_{(x,y) \in A(P)}{w_i(x,y)}=\sum_{j \in J}{a_{i-1,j}}+w_i(x_{i-1},u')+\sum_{(x,y) \in A(P_{a+B})}{w_i(x,y)}=a, 
\end{equation*}
which is a contradiction.

\paragraph{Case 2.} There is no non-trivial subgroup $\{0\}\neq B\le A$ such that $B\subseteq A_{i-1}$. 

\smallskip

By property (i) applied to the set $V$ (which satisfies $|V| \ge n-4|A|$) and to the trivial subgroup $\{0\}<A$ there exist distinct vertices $y_{i-1}, x_i \in V$ such that $a:=w_i(y_{i-1},x_i)\neq 0$. Let us add $y_{i-1},x_i$ to the already constructed sequence $x_0,y_0,x_1,y_1,\ldots,y_{i-2},x_{i-1}$. By the definition of $w_i$, the resulting sequence clearly satisfies the first property we need. For the second property first note that we have $a_{i,j}=w_i(x_{j-1},y_{j-1})+w_i(y_{j-1},x_j)=w_{i-1}(x_{j-1},y_{j-1})+w_{i-1}(y_{j-1},x_j)=a_{i-1,j}$ for all $1 \le j \le i-1$ and $a_{i,i}=w_i(x_{i-1},y_{i-1})+w_i(y_{i-1},x_i)=w_i(y_{i-1},x_i)=a$ which is by assumption non-zero. Recall that $A_i=\{\sum_{j \in J}{a_{i,j}}|J \subseteq \{1,\ldots,i\}\}=A_{i-1}+\{0,a\}$. 
By the inductive assumption on $A_{i-1}$, to obtain $|A_i|\geq i+1$ it suffices to show that $A_i$ is a proper superset of $A_{i-1}$. 

To do so, first note that there has to exist some integer $c\ge 0$ 
such that $ca \notin A_{i-1}$, as otherwise $A_{i-1}$ would contain the cyclic subgroup of $A$ generated by $a$, and hence we would be in Case 1. So let us fix the smallest such integer $c$ (then $c \ge 1$ since $0a=0 \in A_{i-1}$). 
By the minimality of $c$ we have $(c-1)a\in A_{i-1}$ and hence $ca=(c-1)a+a \in A_{i-1}+\{0,a\}=A_i$, which in turn shows $A_i \setminus A_{i-1} \neq \emptyset$, as required. This proves the assertion of the inductive claim, and concludes the proof of our claim.
\end{proof}

To finish the proof of the lemma, we apply the claim with $i=|A|-1$, and obtain a chain of triangles connecting some vertices $x_0$ and $x_{|A|-1}$ and an $A$-arc-labelling $w_{|A|-1}$ switching equivalent to $w$. Note that in this case we necessarily have $A_{|A|-1}=A$, and hence it follows, that inside this chain of triangles there exists for every $a \in A$ a directed path $P_a$ from $x_0$ to $x_{|A|-1}$ whose arc-labels sum up to $a$, i.e. $w_{|A|-1}$ is $A$-complete at $x_0,x_{|A|-1}$. As $w_{|A|-1}$ is switching-equivalent to $w$,  this contradicts property (ii). This final contradiction finishes the proof of the Lemma~\ref{lemma:main}.
\end{proof}

Using Lemma~\ref{lemma:main} we can now give the proof of Theorem~\ref{thm:main1}.

\begin{proof}[Proof of Theorem~\ref{thm:main1}]
We want to show that $n(A) \le 8|A|$ for every non-trivial Abelian group $(A,+)$.

\smallskip

Suppose towards a contradiction that there exists a finite Abelian group $(A,+)$ such that $n(A)>8|A|$ and suppose that the size of $A$ is the smallest possible. By the definition of $n(A)$ there has to exist a zero-sum free $A$-arc-labelling $w$ of $\bivec{K}_{n(A)}$. 
We apply Lemma~\ref{lemma:main} with $n=n(A)$ and the $A$-arc-labelling $w$ to find an $A$-arc-labelling $w'$ of $\bivec{K}_{n(A)}$ that is switching-equivalent to $w$ and one of the following holds.
\begin{enumerate}[label=(\roman*)]
    \item There exists a vertex set $V \subseteq V(\bivec{K}_n)$ such that $|V| \ge n(A)-4|A|>8|A|-4|A|=4|A|$ and a proper subgroup $B<A$ 
    such that $w'(x,y)\in B$ 
    for every $(x,y) \in A(\bivec{K}_n[V])$. 
    \item There exist vertices $u,v\in V(\bivec{K}_n)$ such that $w'$ is $A$-complete at $u,v$. 
\end{enumerate}
Since $w'$ is switching-equivalent to $w$, by Observation~\ref{obs:switching}  $w'$ is also a zero-sum free $A$-arc-labelling of $\bivec{K}_n$. 

\paragraph{Case 1: (i) holds.} 
If $B=\{0\}$ then any digon in $\bivec{K}_n[V]$ would contradict our assumption that $w'$ is a zero-sum free arc-labelling of $\bivec{K}_n$, and hence we may assume $B\neq \{0\}$. 
By the minimality of $|A|$, we must then have $n(B)\leq 8|B|\leq 4|A|<|V|$, where we used that $B$ is a proper subgroup of $A$, and hence $|B|\le|A|/2$. 
Now, by the definition of $n(B)$, $w'$ cannot be a zero-sum-free $B$-labelling of $\bivec{K}_n[V]$, i.e. there has to exist a directed cycle $C$ in $\bivec{K}_n[V]$ whose arc-labels sum up to zero in $B$, and hence in $A$. 
This contradicts the assumption that $w'$ is a zero-sum free $A$-arc-labelling of $\bivec{K}_n$.

\paragraph{Case 2: (ii) holds.} Let $a:=-w'(v,u) \in A$. As $w'$ is $A$-complete at $u,v$ there exists a directed path $P_a$ in $\bivec{K}_n$ from $u$ to $v$ such that $\sum_{(x,y) \in A(P_a)}{w'(x,y)}=a$. Let $C$ be the directed cycle obtained from $P_a$ by adding the arc $(v,u)$. Then we have
\begin{equation*}
 \sum_{(x,y) \in A(C)}{w'(x,y)}=w'(v,u)+\sum_{(x,y) \in A(P_a)}{w'(x,y)}=-a+a=0,   
\end{equation*}
which again contradicts the assumption that $w'$ is a zero-sum free $A$-arc-labelling of $\bivec{K}_n$.

\smallskip

This finishes the proof of Theorem~\ref{thm:main1}.
\end{proof}

We now continue with groups of prime order and prove Theorem~\ref{thm:main2}. It will be convenient to first prove the following auxiliary result.

\begin{lemma}\label{lemma:aux}
Let $p \ge 3$ be a prime number and $w$ a zero-sum free $\mathbb{Z}_p$-arc-labelling of $\bivec{K}_3$. Then there exists a vertex $v \in V(\bivec{K}_n)$ and non-trivial directed paths $P_1, P_2$ in $\bivec{K}_3$ 
ending at $v$ such that the three values $0$, $\sum_{(x,y) \in A(P_1)}{w(x,y)}$, $\sum_{(x,y) \in A(P_2)}{w(x,y)}$ are pairwise distinct. 
\end{lemma}
\begin{proof}
Since $w$ is a zero-sum free arc-labelling, there exists, in particular, no directed cycle 
with all arc labels zero. It follows that we can order the vertices of $\bivec{K}_3$ as $v_1, v_2, v_3$ such that $w(v_i,v_j) \neq 0$ for $1 \le i<j \le 3$. Let $a:=w(v_2,v_3)$. If $w(v_1, v_2)+a \notin \{0,a\}$, then the claim of the lemma is satisfied with $v=v_3$ and the paths $P_1=(v_2,v_3)$, $P_2=(v_1,v_2), (v_2, v_3)$. Hence, moving on, we may assume that $w(v_1,v_2)+a \in \{0,a\}$. As $w(v_1, v_2) \neq 0$, it necessarily follows that $w(v_1,v_2)+a=0$, i.e. $w(v_1, v_2)=-a$. If $w(v_1,v_3) \neq a$ then the claim of the lemma is satisfied with $v=v_3$ and the paths $P_1=(v_1,v_3)$, $P_2=(v_2,v_3)$. Thus, in what follows we may assume that $w(v_1,v_3)=a$. Next, note that $w(v_3,v_2) \neq -a$, as otherwise the digon spanned by the vertices $v_2, v_3$ would form a directed cycle whose arc labels 
sum up to zero. Therefore, if $w(v_3,v_2) \neq 0$ then the statement of the lemma holds with $v=v_2$ and the paths $P_1=(v_1,v_2)$, $P_2=(v_3,v_2)$. 
As a result, we may also assume that $w(v_3,v_2)=0$.
Now, however, it is the vertex $v=v_2$ and the paths $P_1=(v_1,v_3), (v_3,v_2)$, $P_2=(v_1,v_2)$ that fulfill the desired property.

\end{proof}

We are now prepared for the proof of Theorem~\ref{thm:main2}.

\begin{proof}[Proof of Theorem~\ref{thm:main2}]
Suppose towards a contradiction that there exists a prime number $p \ge 3$ such that $n(\mathbb{Z}_p)>\frac{3p-1}{2}$.
Then for $n=\frac{3p-1}{2}$ there must exist a zero-sum free $\mathbb{Z}_p$-arc-labelling $w$ of $\bivec{K}_n$. In order to lead this assumption towards a contradiction, let us first prove the following claim by induction on $i$.

\paragraph{Claim.} For $i=0,1,2,\ldots,\frac{p-1}{2}$ there are distinct vertices $x_0,y_0,z_0,x_1,y_1,z_1,\ldots, x_{i-1}, y_{i-1},z_{i-1}, x_{i}$ in $\bivec{K}_n$ and a $\mathbb{Z}_p$-arc-labelling $w_i$ of $\bivec{K}_n$ switching-equivalent to $w$ such that the following hold for every $1\leq j \leq i$.
\begin{itemize}
\item $w_i(x_{j-1},x_j)=0$. 
\item There exist two directed paths $P_{1,j}, P_{2,j}$ in $\bivec{K}_n[V_j]$, $V_j:=\{x_{j-1},y_{j-1},z_{j-1},x_j\}$ starting at $x_{j-1}$ and ending in $x_j$ such that the three values $0$, $a_{i,j}^{(1)}=\sum_{(x,y) \in A(P_{1,j})}{w_i(x,y)}$ and $a_{i,j}^{(2)}=\sum_{(x,y) \in A(P_{2,j})}{w_i(x,y)}$ are pairwise different.
\end{itemize}

Before moving on to the proof of the claim, note that the setting is analogous to the one in the proof of Theorem~\ref{thm:main1}. However, instead of a chain of triangles we have a chain of $\bivec{K}_4$'s   connecting $x_0$ and $x_i$. When going from $x_0$ to $x_i$, at every intermediate vertex $x_{j-1}$ we have the choice to either go directly to $x_{j}$ and pick up the label $0$, or to take the detour along $P_{1,j}$ or $P_{2,j}$ and pick up labels summing up to $a_{i,j}^{(1)}$ or $a_{i,j}^{(2)}$, respectively. If we put
\begin{equation*}
    A_i:=\{0,a_{i,1}^{(1)},a_{i,1}^{(2)}\}+\{0,a_{i,2}^{(1)},a_{i,2}^{(2)}\}+\cdots+\{0,a_{i,i}^{(1)},a_{i,i}^{(2)}\} \subseteq \mathbb{Z}_p,  
\end{equation*}
then $A_i$ represents  all  possible arc-label sums along possible paths from $x_0$ to $x_i$ in this configuration. 

\begin{proof}[Proof of the claim.]
The claim holds for $i=0$ by simply selecting an arbitrary vertex $x_0 \in V(\bivec{K}_n)$ and putting $w_0:=w$. 
Now let $1 \le i \leq \frac{p-1}{2}$ and suppose the claim holds for $i-1$. 

Put $W=\{x_0,y_0,z_0,x_1,y_1,z_1,\ldots, x_{i-1}, y_{i-2},z_{i-2}, x_{i-1}\}$,  $V=V(\bivec{K}_n)\backslash W$ 
and let $w_i$ be the $\mathbb{Z}_p$-arc-labelling of $\bivec{K}_n$ defined by switching by $w_{i-1}(x_{i-1},y)$ at every $y\in V$. 
Then $w_i$ is clearly switching-equivalent to $w_{i-1}$, and hence to $w$, by definition agrees with $w_{i-1}$ on all arcs spanned by $W$, and $w_i(x_{i-1},y)=0$ for every $y\in V$ . In particular, for $1 \le j \le i-1$ we have $w_i(x_{j-1},x_j)=0$ and  the values $0$, $a_{i,j}^{(1)}=a_{i-1,j}^{(1)}$ and $a_{i,j}^{(2)}=a_{i-1,j}^{(2)}$ are pairwise different.  
Let us now pick three distinct vertices $V_3=\{v_1, v_2, v_3\}\subseteq V$ 
arbitrarily. Since $w$ is a zero-sum free arc-labelling and $w_i$ is switching-equivalent to $w$, by Observation~\ref{obs:switching} $w_i$ is also zero-sum free. In particular, the restriction of $w_i$ to $A(\bivec{K}_n[V_3])$ is also a zero-sum free $\mathbb{Z}_p$-arc-labelling. By the application of Lemma~\ref{lemma:aux} to $\bivec{K}_n[V_3]$ we find that there is a vertex $x_i \in V_3$ and two distinct directed paths $P_1$ and $P_2$ in $\bivec{K}_n[V_3]$ of positive length ending at $x_i$ such that $0$, $b_{i,j}^{1}=\sum_{(x,y) \in A(P_1)}{w_i(x,y)}$ and  $b_{i,j}^{1}=\sum_{(x,y) \in A(P_2)}{w_i(x,y)}$ are pairwise distinct. Let us denote by $u_1$ and $u_2$ the starting vertices of $P_1$ and $P_2$, respectively, and put $\{y_{i-1},z_{i-1}\}=V_3\setminus \{x_i\}$. Furthermore define $P_{1,i}$ and $P_{2,i}$ as the directed $x_{i-1},x_i$-paths in $\bivec{K}_n[V_i]$, $V_i=\{x_{i-1},y_{i-1},z_{i-1},x_i\}$ obtained by extending $P_1$ and $P_2$ with the arcs $(x_{i-1},u_1)$ and $(x_{i-1},u_{2})$, respectively. Then we have $w_i(x_{i-1},x_i)=0$, and the three values $0$, $a_{i,j}^{(1)}=w_i(x_{i-1},u_1)+b_{i,j}^{(1)}=0+b_{i,j}^{(1)}=b_{i,j}^{(1)}$ and $a_{i,j}^{(2)}=w_i(x_{i-1},u_2)+b_{i,j}^{(2)}=0+b_{i,j}^{(2)}=b_{i,j}^{(2)}$ 
are pairwise different, as required.
\end{proof}

To conclude the proof of Theorem~\ref{thm:main2} let 
$x_0,y_0,z_0,x_1,y_1,z_1,\ldots, x_{\frac{p-3}{2}}, y_{\frac{p-3}{2}},z_{\frac{p-3}{2}}, x_{\frac{p-1}{2}}$ and $w_{\frac{p-1}{2}}$ be as given by the above claim when applied with $i=\frac{p-1}{2}$. Let $A_{\frac{p-1}{2}}$ be the set of corresponding path sums, and put $a:=-w_{\frac{p-1}{2}}(x_{\frac{p-1}{2}},x_0)$. 
Since $p$ is a prime number, the iterative application of the Cauchy-Davenport Theorem directly yields that $A_{\frac{p-1}{2}}=\mathbb{Z}_p$ must hold, and hence our chain of $\bivec{K}_4$'s, in particular, contains a path $P_{a}$ from $x_0$ to $x_{\frac{p-1}{2}}$ 
such that $\sum_{(x,y) \in A(P_{a})}{w_{\frac{p-1}{2}}(x,y)}=a$. Let $C$ be the directed cycle obtained from $P_a$ by adding the arc $(x_{\frac{p-1}{2}},x_0)$. 
Then we have 
\begin{equation*}
    \sum_{(x,y)\in A(C)}w_{\frac{p-1}{2}}(x,y)=w_{\frac{p-1}{2}}(x_{\frac{p-1}{2}},x_0) + \sum_{(x,y) \in A(P_{a})}{w_{\frac{p-1}{2}}(x,y)}=-a+a=0.
\end{equation*}
However, this is impossible as $w_{\frac{p-1}{2}}$ is switching equivalent to a zero-sum-free arc-labeling, and as such is zero-sum-free as well.
This final contradiction shows that our very initial assumption that $n(\mathbb{Z}_p)>\frac{3p-1}{2}$ was wrong and hence concludes the proof.
\end{proof}

Finally we give the proof of Corollary~\ref{cor:easy}. Even though the argument is identical to the one in in~\cite{AK20}, we include it for the reader's convenience.

\begin{proof}[Proof of Corollary~\ref{cor:easy}]
Let $q \ge 2$ and $G$ be a $K_{2n(\mathbb{Z}_q)}$-minor. Label the $2n(\mathbb{Z}_q)$ super nodes of this minor as $X_i^+, X_i^-$ for $1 \le i \le n(\mathbb{Z}_q)$. By definition, for every $i$ there is a unique edge $x_i^+x_i^-$ in $G$ connecting a vertex in $X_i^+$ to a vertex in $X_i^-$, and for every $i \neq j$ the induced subgraph $G[X_i^+ \cup X_j^-]$ is a tree. Let us define $w(i,j) \in \mathbb{Z}_q$ to be one plus the length of the unique path connecting $x_i^+$ and $x_j^-$ in this tree taken modulo $q$. This results in a $\mathbb{Z}_q$-arc-labelling of the complete digraph $\bivec{K}_{n(\mathbb{Z}_q)}$ on the vertex-set $\{1,\ldots,n(\mathbb{Z}_q)\}$. Then, by the definition of the function $n(\mathbb{Z}_q)$ 
there has to exist a directed cycle in this auxiliary complete digraph whose arc-labels sum up to zero. Expanding this cycle into a cycle in $G$ by replacing arcs in $\bivec{K}_{n(\mathbb{Z}_q)}$ with the connecting corresponding paths in $G$ yields a cycle of length divisible by $q$ in $G$, as desired.
\end{proof}

\section{Concluding remarks}

In Theorem~\ref{thm:main1} we proved an upper bound on the function $n(A)$ for general finite Abelian groups $A$, which shows that this function grows at most linearly with $|A|$. As demonstrated by Theorem~\ref{thm:main2}, this upper bound can be improved at least when $A=\mathbb{Z}_p$ for some prime $p\geq 3$, but we believe that improvement should be possible in general. To support this belief 
we first remark that, by slightly adjusting the proofs of Lemma~\ref{lemma:main} and 
Theorem~\ref{thm:main1}, 
one can obtain the following result.

\begin{thm}\label{thm:main3}
Let $(A,+)$ be a non-trivial finite Abelian group and let $p$ be the smallest prime divisor of $|A|$. Then we have $n(A)\leq \frac{2p^2}{(p-1)^2}|A|$. 
\end{thm}

The statement of Theorem~\ref{thm:main1} can be recovered by noting that we always have $p\geq2$. 

\smallskip

Next, let us consider Theorem~\ref{thm:main2}. Its proof is based on Lemma~\ref{lemma:aux}, which allowed us to build a long chain of $\bivec{K}_4$'s and hence to prove the $\mathbb{Z}_p$-completeness of the given arc-labelling. An improvement on this auxiliary result would directly improve the upper bound on the function $n(A)$ for groups of prime order. In this direction we propose the following conjecture.

\begin{conj}
Let $t\geq 3$ an integer, $p\geq 3$ a prime number, and $w$ a zero-sum-free $\mathbb{Z}_p$-arc-labelling of $\bivec{K}_t$. Then there exist a vertex $v\in V(\bivec{K}_t)$ and non-trivial directed paths $P_1,P_2,\dots,P_{t-1}$ in $\bivec{K}_t$ ending at $v$ such that the $t-1$ values $\sum_{(x,y)\in A(P_i)}w(x,y)$, $1\leq i \leq t-1$ are all non-zero and pairwise distinct.
\end{conj}

Note that Lemma~\ref{lemma:aux} resolves the case $t=3$, and if the conjecture is true for a given $t$, then, following the arguments from the proof of Theorem~\ref{thm:main2}, one can obtain $n(\mathbb{Z}_p)\leq t\lceil\frac{p-1}{t-1}\rceil+1$ whenever $p\geq 3$ is a prime.

\smallskip

As for lower bounds, it is easy to see that $n(\mathbb{Z}_q) \ge q+1$ for every integer $q \ge 2$. For this, fix some linear ordering $v_1, \ldots, v_{q}$ of the vertices of $\bivec{K}_{q}$ and define a $\mathbb{Z}_q$-arc-labelling of $\bivec{K}_q$ by setting the label of an arc $(v_i,v_j)$ to be equal to $1$ if $i<j$ and $0$ otherwise. Then every directed cycle in $\bivec{K}_q$ uses at least one and at most $q-1$ arcs of the form $(v_i,v_j)$ with $i<j$, and hence the sum of its arc-labels is non-zero in $\mathbb{Z}_q$. Hence, this  arc-labelling is zero sum-free, and shows that $n(\mathbb{Z}_q)>q$, as required.

Surprisingly, for non-cyclic Abelian groups we did not manage to obtain a general linear lower bound in terms of the group order. 

\begin{ques}\label{ourcon}
Do we have $n(A)>|A|$ for every non-trivial finite Abelian group $(A+)$?
\end{ques}

A natural way 
to answer Question~\ref{ourcon} would be via the following product inequality. 
\begin{ques}\label{seconcon}
Is it true that if $(A_1,+)$ and $(A_2,+)$ are non-trivial finite Abelian groups, then 
\begin{equation*}
 n(A_1 \times A_2)-1 \ge (n(A_1)-1)(n(A_2)-1)? 
\end{equation*}
\end{ques}
Let now $(A,+)$ be a non-trivial finite Abelian group and suppose that the answer to Question~\ref{seconcon} is positive. Then, by the Fundamental Theorem on Finite Abelian Groups, $A$ is isomorphic to the direct product $\mathbb{Z}_{q_1} \times \cdots \times \mathbb{Z}_{q_k}$, for some $q_1,\dots,q_k$. Therefore, it follows that $n(A)-1 \ge (n(\mathbb{Z}_{q_1})-1) \cdots (n(\mathbb{Z}_{q_k})-1) \ge q_1 \cdots q_k=|A|$. However, a positive answer to Question~\ref{seconcon} would actually prove much more. Let $\ell \in \mathbb{N}$ be arbitrary and let $A^\ell$ denote the $\ell$-fold direct product of $A$ with itself. Then by Theorem~\ref{thm:main1} we have
\begin{equation*}
(n(A)-1)^\ell \le n(A^\ell)-1 \le 8|A^\ell|-1<8|A|^\ell,
\end{equation*}
i.e. $n(A) \le 8^{1/\ell}|A|+1$. Taking the limit as $\ell \rightarrow \infty$ we would obtain $n(A) \le |A|+1$, which in turn would result in the exact answer $n(A)=|A|+1$.

\smallskip

For small groups 
we managed to check by computer~\cite{man} that $n(A)=|A|+1$ holds for all Abelian groups $(A,+)$ of order at most $6$, and we believe that this equality might actually be true in general. 

\smallskip

As a final comment, let us mention that 
the proof of Lemma~\ref{lemma:main} would still work (without any noteworthy changes), if in the definition of $A$-completeness we would require paths of \emph{length at least two}. 
Using this stronger statement in the proof of Theorem~\ref{thm:main1} one can show that in every labelling of the complete digraph on $n \ge 8|A|$ vertices with labels from an Abelian group $(A,+)$ there exists a directed cycle \emph{of length at least three} whose arc-labels sum up to zero. This slightly stronger statement has the following immediate consequence concerning the undirected version of the zero-sum cycle problem. 

\begin{corollary}
Let $(A,+)$ be a finite Abelian group, $n\ge 8|A|$, and let $w:E(K_n) \rightarrow A$ be an edge-labelling of the complete undirected graph of order $n$. Then there exists a cycle $C$ in $K_n$ such that $\sum_{e \in E(C)}{w(e)}=0$. 
\end{corollary}

\paragraph{Acknowledgement.} We would like to thank Manfred Scheucher very much for implementing and solving a SAT-model~\cite{man} to compute the values of $n(A)$ for groups of small order.

\end{document}